\documentclass[12pt]{amsart}
\usepackage[toc,page]{appendix}
\usepackage{amsmath}
\usepackage{fullpage}
\usepackage{xspace}
\usepackage[psamsfonts]{amssymb}
\usepackage[latin1]{inputenc}
\usepackage{graphicx,color}
\input xy
\xyoption{all}
\usepackage{wrapfig}

\usepackage{hyperref}
\usepackage{graphicx}

\theoremstyle{plain}

\def\To_#1{\mathrel{\mathop{\longrightarrow}\limits_{#1}}}
\def\L{{\mathrm{L}}}

\def\S{{\mathrm{S}}}

\def\P{{\mathbb{P}}}
\def\C{{\mathbb{C}}}
\def\T{{\mathcal T}}
\def\M{{\mathcal M}}
\def\W{{\mathcal W}}

\def\E{{\mathrm E}}
\def\V{{\mathcal V}}

\def\hW{\widehat\W}
\def\hV{\widehat\V}
\def\hSL{\widehat\S}

\def\hE{\widehat \E}

\def\hS{\widehat \S}
\def\bp{\circledast}
\def\mod{\mathrm {Mod}}
\def\Def{\mathrm {Def}}
\def\proj{\mathrm{proj}}
\newtheorem{theorem}{\bf Theorem}[section]
\newtheorem{proposition}{\bf Proposition}[section]

\newtheorem{corollary}[proposition]{\bf Corollary}

\newtheorem{remark}[proposition]{\bf Remark}

\title{A disconnected deformation space of rational maps}
\begin{author}[E. Hironaka]{Eriko Hironaka}\thanks{The first author was partially supported by a collaboration grant from the Simons Foundation \#209171}
\email{hironaka@math.fsu.edu }
\address{ %
Department of Mathematics\\
Florida State University\\
1017 Academic Way, 208 LOV \\
Tallahassee, FL 32306-4510}
\end{author}
\begin{author}[S. Koch]{Sarah Koch}\thanks{The research of the second author was supported in part by the NSF and the Sloan Foundation}
\email{kochsc@umich.edu}
\address{ %
Department of Mathematics\\
University of Michigan\\
East Hall, 530 Church Street \\
Ann Arbor, Michigan 48109}
\end{author}

\begin{document}

\begin{abstract}
Let $f:(\P^1,P)\to(\P^1,P)$ be a postcritically finite rational map with postcritical set $P$. William Thurston showed that $f$ induces a holomorphic pullback map $\sigma_f:\T_P\to\T_P$ on the Teichm\"uller space ${\T}_P:=\mathrm{Teich}(\P^1,P)$. If $f$ is not a flexible Latt\`es map, Thurston proved 
that $\sigma_f$ has a unique fixed point. 
 In his PhD thesis, Adam Epstein 
generalized  Thurston's ideas and defined a deformation space associated to a rational map $f:(\P^1,A)\to (\P^1,B)$ where $A \subseteq B$, 
allowing for maps $f$ which are not necessarily postcritically finite. By definition, the deformation space $\Def_B^A(f)\subseteq \T_B$ is the locus where the pullback map $\sigma_f:\T_B\to\T_A$ and the forgetful map $\sigma_A^B:\T_B\to\T_A$ agree. 
Using purely local arguments, Epstein showed that $\Def_B^A(f)$  is a smooth 
analytic submanifold of $\T_B$ of dimension $|B-A|$. In this article, we investigate the question of whether $\Def_B^A(f)$ is connected. We exhibit a family of quadratic rational maps for which the associated deformation spaces are disconnected; in fact, each has infinitely many components. 

\end{abstract}
\maketitle
\section{Introduction}

\noindent Let $f:\P^1\to\P^1$ be a rational map of degree at least $2$ with critical set $\Omega$. The postcritical set of $f$ is $P:=\cup_{n>0}f^n(\Omega)$,  the smallest forward invariant set that contains the critical values of $f$.  We regard $f$ as a self-map of the pointed space $(\P^1,P)$, and  $f$ is {\em postcritically finite} if $|P|<\infty$. In the 1980s, William Thurston established a topological characterization of rational maps which are postcritically finite \cite{dh}. The proof of this theorem associates a dynamical system $\sigma_f:\T_P\to\T_P$ to the dynamical system $f:(\P^1,P)\to (\P^1,P)$, where the space $\T_P$ is the {Teichm\"uller space} of the pair $(\P^1,P)$, and the map $\sigma_f:\T_P\to \T_P$ is a holomorphic pullback map. 

The map $\sigma_f:\T_P\to \T_P$ has a fixed point, and if $f:(\P^1,P)\to (\P^1,P)$ is not of {Latt\`es type}, Thurston proved that some iterate of $\sigma_f:\T_P\to\T_P$ is a strict contraction. The fixed point of $\sigma_f:\T_P\to\T_P$ is therefore unique. This fact is known as {\em Thurston rigidity}. 

\begin{theorem}[W. Thurston, \cite{dh}]\label{rigidity}
If $f:(\P^1,P)\to (\P^1,P)$ is a postcritically finite rational map which is not of Latt\`es type, then the associated pullback map $\sigma_f:\T_P\to \T_P$ has a unique fixed point. 
\end{theorem} 

\noindent In 1993, Adam Epstein generalized Thurston's ideas to rational maps $f:(\P^1,A)\to (\P^1,B)$ which are not necessarily postcritically finite \cite{adam}. In this setting, he imposed the following conditions on the sets $A$ and $B$: 
\begin{enumerate}
\item $A$ and $B$ are finite, each containing at least 3 points,
\item $B$ contains the critical values of $f$,
\item $f(A)\subseteq B$, and 
\item $A\subseteq B$.
\end{enumerate}
The first three conditions are necessary to define a  pullback map $\sigma_f:\T_B\to \T_A$, and the last condition is necessary to define a {forgetful map} $\sigma_A^B:\T_B\to\T_A$ (see Section \ref{prelims} for definitions).

\subsection{The deformation space} Let $f:(\P^1,A)\to (\P^1,B)$ be a rational map for which the sets $A$ and $B$ satisfy the conditions (1)-(4) above. Epstein defined the associated {\em deformation space} 
\[
\Def_B^A(f):=\{\tau\in\T_B\;|\;\sigma_f(\tau)=\sigma_A^B(\tau)\}. 
\] 
In other words, $\Def_B^A(f)$ is the {\em equalizer} of the two maps $\sigma_f:\T_B\to\T_A$ and $\sigma_A^B:\T_B\to\T_A$. 
\begin{theorem}[A. Epstein, \cite{adam}]\label{localDef}
Let $f:(\P^1,A)\to (\P^1,B)$ be a rational map that is not of Latt\`es type. The deformation space $\Def_B^A(f)$ is a smooth analytic submanifold of $\T_B$ of dimension $|B-A|$. 
\end{theorem}
\noindent If $A=B$ in the theorem above, then $\Def_B^A(f)$ is equal to the set of fixed points of $\sigma_f$, so by Theorem \ref{rigidity}, $\Def_B^A(f)$  consists of a single point \cite{bct}. 

The proof of Theorem \ref{localDef} is purely local and reveals nothing about the global structure of $\Def_B^A(f)$. In this article, we prove that $\Def_B^A(f)$ is disconnected for the following quadratic rational maps.

\subsection {A family of quadratic rational maps} Let $\mathrm{M}_2$ be the moduli space of quadratic rational maps: the space of quadratic rational maps up to conjugation by M\"obius transformations. The locus $\mathrm{Per}_4(0)\subseteq \mathrm{M}_2$ consists of  conjugacy classes of  maps with a superattracting cycle of period $4$. That is, for $\langle f\rangle \in \mathrm{Per}_4(0)$  the map $f:\P^1\to\P^1$ has a critical point which is periodic of period $4$. Define $\mathrm{Per}_4(0)^*\subseteq \mathrm{Per}_4(0)$ to be the subset consisting of $\langle f \rangle$ for which the superattracting $4$-cycle contains only one critical point of $f$. 

Let $f:\P^1\to\P^1$ represent an element of $\mathrm{Per}_4(0)^*$. Define the set $A$ to be the set of  points in the superattracting $4$-cycle, and define the set $B:=A\cup \mathrm{cv}(f)$, where $\mathrm{cv}(f)$ is the set of critical values of $f$. Consider the associated deformation space $\Def_B^A(f)$. By Theorem~\ref{localDef}, $\Def_B^A(f)$ is a 1-dimensional  submanifold of $\T_B$ which is 2-dimensional. The following theorem is our main result.

\begin{theorem}[Main Theorem]\label{mainthm}
For $\langle f\rangle \in\mathrm{Per}_4(0)^* $, the space $\Def_B^A(f)$ has infinitely many connected components. 
\end{theorem}

\noindent 
The proof of Theorem~\ref{mainthm}  reduces to a comparison of the stabilizer of $\Def_B^A(f)$ in the automorphism group
of Teichm\"uller space with the stabilizer of a connected component of $\Def_B^A(f)$. 

The Teichm\"uller space $\T_B$ is the universal cover of the {moduli space} $\M_B$ (the moduli space of ordered points on the Riemann sphere), and the group of deck transformations of the universal covering map $\T_B\to \M_B$ is naturally isomorphic to $\mod_B$, the {pure mapping class group} of $(\P^1,B)$. We introduce a certain subgroup $\S_f\subseteq \mod_B$ called the {special liftables}.  This group, which is the trivial group in the dynamical setting where $A=B$, plays a fundamental role in our study of $\Def_B^A(f)$. In Proposition~\ref{SProp}, we show that $\S_f$ consists of elements which restrict to automorphisms of $\Def_B^A(f)$. The space $\T_B$ comes with a canonical basepoint $\bp\in \Def_B^A(f)$. Let $\overline \bp$ be the image of $\bp$ under the quotient map $\Def_B^A(f)\to\Def_B^A(f)/\S_f$, and consider the homomorphism 
\[
\overline \iota_*: \pi_1\left(\Def_B^A(f)/\S_f,\overline \bp\right)\longrightarrow \pi_1\left(\T_B/\S_f,\overline\bp\right) 
\]
induced by inclusion. We establish Theorem~\ref{mainthm} by showing that the index of the image of $\overline \iota_*$ in $\pi_1\left(\T_B/\S_f, \overline\bp\right)$ is infinite.

\begin{remark} {\em The authors learned that T. Firsova, J. Kahn, and N. Selinger proved a related result in \cite{tanya}. Their work was completed independently, and at the same time that the authors proved Theorem \ref{mainthm}.}
\end{remark}

\noindent {\bf Outline.} The paper is organized as follows.  In Section~\ref{prelims}, we recall basic definitions and introduce subgroups 
$\S_f \subseteq \L_f \subseteq \mod_B$ associated to $\Def_B^A(f)$.  In Section~\ref{quotients}, we look at the image of $\Def_B^A(f)$ in 
an intermediate covering space $\W$ of $\M_B$, and reduce the problem of connectivity to a question about covering spaces and group actions.
Theorem~\ref{mainthm} is proved in Section~\ref{proof}.   Relevant general techniques for computing the fundamental group of complements of
plane curves, in particular real line arrangements, are recalled in Appendix~\ref{zvk-sec}.
\smallskip 

\noindent{\bf Acknowledgments.} We would like to thank Matthieu Astorg, Laurent Bartholdi, Adam Epstein, John Hubbard, and Curtis McMullen for helpful conversations related to this work. 

\section{Preliminaries}\label{prelims}

\noindent In this section, we review definitions and translate the problem of connectivity of $\Def_B^A(f)$ to a problem about groups.

Recall that a Riemann surface is a connected oriented topological
surface together with a {\em complex structure}:  a
maximal atlas of charts $\phi:U\to\C$ with holomorphic overlap maps.
For a given oriented, compact topological surface $X$, we denote the
set of all complex structures on $X$ by ${\mathcal{C}}(X)$. An orientation-preserving branched covering map $f:X\to
Y$ induces a map $f^*: {\mathcal{C}}(Y)\to {\mathcal{C}}(X)$; in particular,
for any orientation-preserving homeomorphism $\psi:X\to X$, there is an induced
map $\psi^*: {\mathcal{C}}(X)\to {\mathcal{C}}(X)$.

Let $A\subseteq X$ be finite. The Teichm\"uller space of $(X,A)$ is
\[{\mathrm {Teich}}(X,A):= {\mathcal{C}}(X)/{\sim_A}\]
where $c_1\sim_A c_2$ if and only if $c_1=\psi^*(c_2)$ for some
orientation-preserving homeomorphism $\psi:X\to X$ which is isotopic
to the identity relative to $A$.

\subsection{The forgetful map} If $A\subseteq B\subseteq X$ are finite sets then for $c_1,c_2\in {\mathcal C}(X)$, 
if  $c_1\sim_B c_2$ then  $c_1\sim_A c_2$.  
This allows us to define the {\em forgetful map} $\sigma_A^B:\mathrm{Teich}(X,B)\to \mathrm{Teich}(X,A)$ taking
the equivalence class of $c$ in $\mathrm{Teich}(X,B)$ to the equivalence class of $c$ in $\mathrm{Teich}(X,A)$.

\subsection{The pullback map} 
In view of the homotopy-lifting property, if
\begin{itemize}
\item $B\subseteq Y$ is finite and contains the critical values of $f$, and
\item $f(A)\subseteq B$,
\end{itemize}
then $f^*:{\mathcal{C}}(Y)\to{\mathcal{C}}(X)$ descends to a well-defined
map $\sigma_f$ between the corresponding Teichm\"uller spaces.
\[
\xymatrix{
 {\mathcal{C}}(Y)\ar[d]\ar[r]^{f^*} & {\mathcal{C}}(X) \ar[d] \\
 \text{Teich}(Y,B)\ar[r]^{\sigma_f} & \text{Teich}(X,A) }\]
 This map is known as the {\em pullback map} induced by $f$.

\subsection{Genus zero} Let $X$ be the 2-sphere and fix an identification $X = \P^1$.
Then the Teichm\"uller space $\T_A := \mathrm{Teich}(X,A)$ comes with a 
{\it canonical basepoint} $\bp:=[\mathrm{id}]$, and we use the 
Uniformization Theorem to
obtain the following description of $\T_A$. Given a finite set $A\subseteq \P^1$ containing at least 3 points,  the space $\T_A$ is the quotient of the
space of all orientation-preserving homeomorphisms  $\phi :
\P^1\rightarrow \P^1$ by the equivalence relation $\sim$ where
$\phi_1\sim\phi_2$  if there exists a M\"obius transformation $\nu$
such that $\nu\circ\phi_1=\phi_2$ on $A$, and $\nu\circ\phi_1$ is
isotopic to $\phi_2$ relative to $A$. The Teichm\"uller space  $\T_A$ is a complex manifold of dimension $|A|-3$. 

The {\em moduli space} of the pair $(\P^1,A)$ is the space of injective maps $\varphi:A\hookrightarrow \P^1$ modulo postcomposition by M\"obius transformations. The moduli space is  a complex manifold isomorphic to the complement of finitely many hyperplanes in $\C^{|A|-3}$. We denote the moduli space as $\M_A$. 

If $\phi$ represents an element of the Teichm\"uller space $\T_A$, the restriction $\phi\mapsto \phi|_A$ induces a universal covering map $\T_A\to \M_A$ which is a local biholomorphism with respect to the complex structures on $\T_A$ and $\M_A$. The group of deck transformations is naturally isomorphic to the {\em pure mapping class group}  $\mod_A$, the quotient of the group of orientation-preserving homeomorphisms $h:\P^1\to \P^1$ fixing $A$ pointwise by the subgroup of such maps isotopic to the identity relative to $A$. This group acts freely and properly discontinuously on $\T_A$. 

The forgetful map $\sigma_A^B:\T_B\to\T_A$ is a holomorphic surjective submersion and descends to the corresponding {forgetful map} on moduli space $\mu_A^B:\M_B\to \M_A$. 

\subsection{Admissible rational maps} We will say that the rational map  $f:(\P^1,A)\to (\P^1,B)$ is {\em admissible} if:
\begin{itemize}
\item $A$ and $B$ are finite sets, each containing at least 3 points,
\item $B$ contains the critical values of $f$, and
\item $f(A)\subseteq B$.
\end{itemize}
Let  $f:(\P^1,A)\to (\P^1,B)$  be an admissible rational map, let $\tau\in \T_B$ and let $\phi:\P^1\to \P^1$ be a homeomorphism
representing $\tau$. By the
Uniformization Theorem, there exist
\begin{itemize}
\item a 
homeomorphism $\psi:\P^1\to \P^1$ representing $\tau':=
\sigma_f(\tau)$, and
\item  a  rational map $F:\P^1\to \P^1$,
\end{itemize}
such that the following diagram commutes.
\[\xymatrix{
(\P^1,A)\ar[d]^{f}\ar[rr]^{\psi} &&
\bigl(\P^1,\psi(A)\bigr) \ar[d]^{F} \\
(\P^1,B)\ar[rr]^{\phi} && \bigl(\P^1,\phi(B)\bigr) }
\]
Conversely, if we have such a commutative diagram with $F$
holomorphic, then
\[\sigma_f(\tau)=\tau'\]
where $\tau\in\T_B$ and $\tau'\in \T_A$ are the
equivalence classes of $\phi$ and $\psi$ respectively. The map $\sigma_f:\T_B\to\T_A$ is holomorphic. 

\subsection{The liftables.}\label{liftable}  Let $f:(\P^1,A)\to (\P^1,B)$ be an admissible rational map, and let \linebreak $h:(\P^1,B)\to (\P^1,B)$ be an orientation-preserving homeomorphism that fixes $B$ pointwise. 
We say that $h$ is {\em liftable} if there is an orientation-preserving homeomorphism $h':(\P^1,A)\to(\P^1,A)$ fixing $A$ pointwise so that $h\circ f = f\circ h'$.  The lift $h'$ is only determined up to deck transformations of $f$
that preserve $A$.
Let
 \[
\L_f:=\{[h]\in\mod_B\;|\;h\text{ is liftable}\}
\]
be the subgroup of {\em liftable mapping classes}, or just the {\em liftables} associated to $f$.
As proved in Proposition 3.1 of \cite{kps}, the subgroup $\L_f$ has finite index in $\mod_B$, and there is a {\em lifting homomorphism}
\begin{eqnarray*}
\Phi_f:\L_f &\rightarrow&\mod_A\\
{[h]} &\mapsto& {[h']}
\end{eqnarray*}
which is well-defined because the action of $\mod_A$ is free on $\T_A$. 
%

\subsection{The special liftables.}\label{special} Let $f:(\P^1,A)\to (\P^1,B)$ be an admissible rational map. We now work in the case where domain and range of $f$ are identified and $A\subseteq B$. We define a subgroup $\S_f \subseteq \L_f$ that preserves $\Def_B^A(f)$.

Because $A\subseteq B$, there is a {\em forgetful homomorphism} $\Phi_A^B:\mod_B\to\mod_A$ corresponding to forgetting points in $B-A$. The subgroup of {\em special liftable mapping classes}, or just the {\em special liftables} associated to $f:(\P^1,A)\to (\P^1,B)$ is defined to be
\[
\S_f:=\{g\in \L_f\;|\;\Phi_f(g)=\Phi_A^B(g)\}, 
\]
the equalizer of the homomorphisms $\Phi_f$ and $\Phi_A^B$. Not much is known about this subgroup in general. In the purely dynamical setting where $A=B$ and $f$ is not a Latt\`es map, it follows from Theorem \ref{rigidity} that $\S_f$ is the trivial group; in particular, it has infinite index in $\mod_B$. 

\begin{proposition}\label{SProp} Let $g\in \L_f$. The following are equivalent:
\begin{enumerate}
\item $g\cdot\Def_B^A(f)\cap \Def_B^A(f)\neq \emptyset$,
\item $g\cdot \Def_B^A(f)=\Def_B^A(f)$, and 
\item $g\in \S_f$.
\end{enumerate}
\end{proposition}
\begin{proof}
The proof is purely formal and uses the following two equations:
\begin{itemize}
\item\label{eqn1} for all $h \in \mod_B$ and for all $\tau\in \T_B$,
\[
\sigma_A^B(h\cdot \tau)=\Phi_A^B(h)\cdot\sigma_A^B(\tau),
\]
\item\label{eqn2} for all $h\in \L_f$ and  for all $\tau\in\T_B$,
\[
\sigma_f(h\cdot \tau)=\Phi_f(h)\cdot \sigma_f(\tau).
\]
\end{itemize}
Let $\tau \in \Def_B^A(f)$. Since $g\in \L_f$, we have 
\[
\sigma_f(g\cdot \tau)=\sigma_A^B(g\cdot \tau) \quad \text{if and only if}\quad
\Phi_f(g)\cdot\sigma_f(\tau) =\Phi_A^B(g)\cdot\sigma_A^B(\tau).
\]
Because $\tau\in\Def_B^A(f)$, $\sigma_f(\tau)=\sigma_A^B(\tau)$.   Since elements of $\mod_B$ have no fixed points,
we have, for any $\tau' \in \T_A$,
\[
\Phi_f(g)\cdot\tau' =\Phi_A^B(g)\cdot\tau' \quad \text{ if and only if} \quad \Phi_f(g)=\Phi_A^B(g) \quad \text{ if and only if} \quad
g\in\S_f.
\]
\end{proof}

\section{Quotients}\label{quotients}

\noindent 
Throughout the rest of this paper, let $f:(\P^1,A)\to (\P^1,B)$ be an admissible rational map so that $A\subseteq B$, and 
assume that $f$ is not of Latt\`es type. 

The quotient $\W:=\T_B/\L_f$ is a connected complex manifold of dimension $|B|-3$, and the quotient $\Def_B^A(f)/\S_f$ is a (possibly disconnected) complex submanifold of $\T_B/\S_f$ of dimension $|B-A|$.  The space $\W$ comes equipped with maps 
$$
\mu_B:\W\to \M_B\quad \text{and} \quad \mu_A:\W\to \M_A
$$
 so that the diagram below, excluding the dashed arrow, commutes. 

\begin{eqnarray}\label{bigdiagram}
\xymatrix{
\Def_B^A(f) \ar @{^{(}->}[r]^\iota\ar[d] &\T_B\ar[rr]^{\sigma_f}\ar[d]^q & &\T_A\ar[ddd]\\
\Def_B^A(f)/\S_f  \ar @{^{(}->}[r]^{\overline{\iota}}  &{\T_B/\S_f}\ar[d]^\omega  & &\\
 &\W:=\T_B/\L_f\ar[d]^{\mu_B}\ar[rrd]^{\mu_A} & &\\
 &\M_B \ar@{-->}[rr]_{\mu_A^B}&&\M_A}
\end{eqnarray}

\noindent The map $\mu_B:\W\to\M_B$ is a finite cover and the map $\mu_A:\W\to\M_A$ can be just about anything; for example, it may be constant \cite{bekp}. 
It follows from Proposition \ref{SProp}, that $\omega\circ \overline\iota$ is injective. Let $\V\subseteq \W$ denote its image. The space $\V$ is a subset of the equalizer of the pair of maps $\mu_A:\W\to\M_A$ and $\mu_A^B \circ \mu_B:\W\to\M_A$; that is 
\begin{eqnarray}\label{diag}
\V\subseteq \{w\in \W\;|\;\mu_A(w)=\mu_A^B \circ \mu_B(w)\}.
\end{eqnarray}

\subsection{Fundamental groups} In this section, we review some notions from the theory of covering spaces that will be required for the proof of Theorem \ref{mainthm}. 

 The canonical
basepoint $\bp$ of $\T_B$ lies in $\Def_B^A(f)$ and determines 
basepoints
\[
\bp_\V\in \V\subseteq \W,\quad \bp_B\in\M_B,\quad \text{and}\quad \bp_A\in\M_A.
\]
\noindent Let $i:\V\hookrightarrow \W$  be the inclusion. We have the following commutative diagram.
\[
\xymatrix{
(\Def_B^A(f)/\S_f,\overline\bp)\ar[d]^{\omega\circ \overline\iota}\ar @{^{(}->}[r]^{\overline\iota} &(\T_B/\S_f,\overline\bp)\ar[d]^\omega\\
(\V,\bp_\V)\ar @{^{(}->}[r]^i &(\W,\bp_\V)
}
\]
It follows from Proposition \ref{SProp} the map $\omega\circ\overline\iota$ is a homeomorphism.   

The induced diagram on fundamental groups is:
\[
\xymatrix{
\pi_1(\Def_B^A(f)/\S_f,\overline\bp)\ar[r]^{\overline\iota_*}\ar[d]^{({\omega\circ \overline\iota})_*} &\pi_1(\T_B/\S_f,\overline\bp)\ar[d]^{\omega_*}\\
\pi_1(\V,\bp_\V)\ar[r]^{i_*} &\pi_1(\W,\bp_\V)
}
\]
where $(\omega\circ \overline\iota)_*$ is an isomorphism.
If $\V$ is disconnected, $\pi_1(\V,\bp_\V)$ is defined to be the fundamental group of the connected component containing $\bp_\V$. 

As a consequence we have the following result. 
 \begin{proposition}
 The map $\overline\iota_*$ is injective (respectively surjective) if and only if $i_*$ is injective (respectively surjective). 
 \end{proposition}
\begin{proposition} \label{diag1} The basepoints determined by $\bp$ define identifications
\begin{eqnarray*}
\mod_B &=& \pi_1(\M_B,\bp_B)\\
\mod_A &=&  \pi_1(\M_A,\bp_A)\\
 \L_f &=& \pi_1(\W,\bp_\V)  \\
  \S_f   &=& \omega_*( \pi_1(\T_B/\S_f,\overline\bp)) \subseteq \L_f\\
\end{eqnarray*}
such that 
\[
\Phi_f = (\mu_A)_*,\quad  \Phi_A^B = (\mu_A^B)_*,\quad \text{and}\quad \Phi_A^B|_{\L_f} = \left(\mu_A^B \circ \mu_B\right)_*.
\]
\end{proposition}
\begin{proof}
The proofs that 
\[
\Phi_A^B = (\mu_A^B)_*\quad \text{and}\quad \Phi_A^B|_{\L_f} = \left(\mu_A^B \circ \mu_B\right)_*
\]
are immediate from the definitions and from the identifications above. We will  prove that $\Phi_f = (\mu_A)_*$. Let $\gamma$ be an oriented loop in $\W$ based at $\bp_\V$, and set $\alpha:=\mu_A(\gamma)$, an oriented loop in $\M_A$ based at $\bp_A$. By the identifications above, $[\gamma]\in \pi_1(\W,\bp_\V)$ determines some $g\in \L_f$, and $[\alpha]\in\pi_1(\M_A,\bp_A)$ determines some $h\in \mod_A$. Because Diagram (\ref{bigdiagram}) commutes, we have 
\[
\sigma_f(g\cdot \bp)=h\cdot \sigma_f(\bp).
\]
Because $g\in \L_f$, we must have $h=\Phi_f(g)$. 
\end{proof}

\begin{proposition}\label{diag2} We have 
\[
\S_f
=\{\gamma\in \L_f\;|\; (\mu_A)_*(\gamma)=\left(\mu_A^B \circ \mu_B\right)_*(\gamma)\}.
\]
\end{proposition}
\begin{proof}
This is a direct consequence of Proposition \ref{diag1}. 
\end{proof}

\subsection{A criterion for connectivity}
In our particular family of examples, $\V$ is connected (see Section~\ref{spaceV}), and hence the question arises whether or not
$\Def_B^A(f)$ is also connected.
Let 
$$
\E_f:=i_*\left(\pi_1(\V,\bp_\V)\right)\subseteq \S_f.
$$
In general we have the following.

\begin{proposition}\label{cosets-prop}  Suppose that $\V$ is connected. There is a bijection between the connected components of $\Def_B^A(f)$ and the (left) cosets of $\E_f$ in $\S_f$.
\end{proposition}

\begin{proof} The fibers of the covering map  $p:=\omega \circ q\circ \iota: \Def_B^A(f) \rightarrow  \V$ 
are in bijective
correspondence with $\S_f$.  In particular, path-lifting at $\bp$ defines a bijection
$$
\beta_S: \S_f \rightarrow p^{-1}(\bp_\V),
$$
and an injection
$$
\beta_E: \E_f \rightarrow p^{-1}(\bp_\V),
$$
whose image is contained in a connected component of $p^{-1}(\V)$.

For two elements of $\gamma,\gamma' \in \S_f$, $\beta(\gamma)$ and $\beta(\gamma')$ lie in the same
component of $\Def_B^A(f)$ if and only if $\gamma' = \gamma \circ \alpha$ for some $\alpha \in \E_f$,
that is, if and only if the cosets are equal:
$$
\gamma' \cdot \E_f = \gamma\cdot \E_f.
$$
\end{proof}

\noindent
As an easy consequence, we have the following.
\begin{corollary}\label{cosets}
Suppose that $\V$ is connected. Then $\Def_B^A(f)$ is connected if and only if $\E_f=\S_f$. 
\end{corollary}

\noindent 
It is also of interest to consider the topology of the components of $\Def_B^A(f)$.
\begin{proposition}
The components of $\Def_B^A(f)$ are simply-connected if and only if $i_*$ is injective.
\end{proposition}

\begin{proof} By the theory of covering spaces, each connected component of $\Def_B^A(f)$,  is a covering space of $\V$ with fundamental group 
equal to the kernel of $i_*$.
\end{proof}

\subsection{Equalizers and fundamental groups} We pause here to reformulate the problem we are solving for our family of maps $f:(\P^1,A)\to(\P^1,B)$. In Proposition~\ref{spaceV}, we will prove that 
\[
\V= \{w \in \W \ | \ \mu_A^B \circ \mu_B (w) = \mu_A(w)\}
\]
that is, $\V$ is the equalizer of the maps $\mu_A^B \circ \mu_B:\W\to\M_A$ and $\mu_A:\W\to\M_A$. By Proposition \ref{diag2}, we have 
\[
\S_f= \{\gamma \in \pi_1(\W,\bp_\V) \ | \ (\mu_A^B \circ \mu_B)_*(\gamma) = (\mu_A)_*(\gamma)\}.
\]
The space $\Def_A^B(f)$ is connected if and only if $\E_f=\S_f$.  That is, if the fundamental group of the equalizer equals the equalizer of the
fundamental groups.
For our particular examples $f:(\P^1,A)\to (\P^1,B)$, we will prove that that $\E_f$ has infinite index in $\S_f$. By Proposition \ref{cosets-prop}, this will establish Theorem \ref{mainthm}. 

\section{The proof of Theorem \ref{mainthm}}\label{proof}

\noindent Let $\langle f\rangle \in \mathrm{Per}_4(0)^*$. By conjugating with a M\"obius transformation, we may suppose that $f$ has a superattracting cycle of the form 
\begin{eqnarray}\label{cycle-eqn}
\xymatrix{
&0 \ar[r]^{2} & \infty \ar[r] & 1 \ar[r] &a \ar@(dl,dr)[lll] }
\end{eqnarray}
\smallskip

\noindent As evident by the work that follows, our results hold for any map $f$ representing an element of $\mathrm{Per}_4(0)^*$ with appropriately defined sets $A$ and $B$. For the sake of presentation however, we will work with a concrete example, so that in our coordinates the basepoint $\bp_\V\in\mathbb{R}$. We will work with
\[
f:\P^1\to\P^1\quad\text{given by}\quad f:z\mapsto \frac{(4z-3)(z+2)}{4z^2}.
\]
 The map $f$ has a superattracting cycle 
of the form in Line (\ref{cycle-eqn}) for $a =3/4$. The critical points of $f$ are $\{0,12/5\}$, and the critical values of $f$ are $\{\infty,121/96\}$. Define the set $A=\{0,1,\infty,3/4\}$, and the set $B=A\cup \{121/96\}$. By Theorem~\ref{localDef}, $\Def_B^A(f)$ is a 1-dimensional submanifold of a 2-dimensional Teichm\"uller space $\T_B$. The space $\T_A$ is 1-dimensional. We compute the spaces $\W$ and $\V$ for this particular $f$.  

We set our coordinates before doing  computations. Given $[\alpha] \in \M_A$ and 
$[\beta] \in \M_B$, 
let $\nu_\alpha$ and $\nu_\beta$ be M\"obius transformations such that
$\nu_\alpha \circ \alpha$ and $\nu_\beta \circ \beta$ are the identity on $\{0,1,\infty\}$.
Then we define coordinates:
\begin{eqnarray*}
\M_A  &\rightarrow& \C -\{0,1\}\\
\left [\alpha \right ] &\mapsto& x := (\nu_\alpha \circ \alpha)(3/4),
\end{eqnarray*}
and
\begin{eqnarray*}
\M_B  &\rightarrow& \mathbb{C}^2-\{y=0,y=1,y=z,z=0,z=1\}\\
\left [\beta \right ]&\mapsto& (y,z):=((\nu_{\beta} \circ \beta)(3/4),(\nu_\beta \circ \beta)(121/96)).
\end{eqnarray*}

\subsection{The space $\W$} By Lemma 2.5 in \cite{sarah}, a point in $\W$ consists of 
$(x,y,z,F)$ where $F$ is a rational map
$$
F:(\P^1,\{0,1,\infty,x\})\to (\P^1,\{0,1,\infty,y,z\})
$$
satisfying the combinatorial conditions below, where the marked points are distinct: 
\[
\xymatrix{
0\ar[d]^2 &\infty\ar[d]  &  1\ar[d]    &   x\ar[d]       & \ast\ar[d]^2  \\
\infty &1                    & y                & 0              & z}
\]
that is, $0$ is a critical point of $F$, $\mathrm{cv}(F)=\{\infty,z\}$, and 
\[
F(0)=\infty, \quad F(\infty)=1,\quad F(1)=y,\quad \text{and}\quad F(x)=0.
\]
As can easily be verified, such a rational map $F:\P^1\to \P^1$ must be of the following form: 
\[
F(t)=\frac{(x-t)(-tx+y+t+x-1)}{(x-1)t^2},\quad\text{where}\quad z=\frac{(-x^2+y+2x-1)^2}{4x(y-1+x)(1-x)}. 
\]
Note that the map $F$ has a superattracting cycle of the form in Line (\ref{cycle-eqn}) if and only if $x=y$. 

There is an isomorphism 
\[
\W\to\C^2-\Delta \quad \text{given by} \quad (x,y,z,F)\mapsto (x,y) 
\]
where $\Delta$ consists of all ``forbidden'' pairs $(x,y)$ leading to collisions of points in $\{0,1,\infty,z\}$, or collisions of points in $\{0,1,\infty,y,z\}$. The set $\Delta$ can be computed explicitly:
\[
\Delta=\{(x,y)\in\mathbb{C}^2\;|\;x=0,\;y=0,\;y=1,\;x=1,\;y-1+x=0,\;x^2-y-2x+1=0,
\]
\[
x^2+y-1=0,\text{ or }2xy+x^2-y-2x+1=0\}.
\]
We will use $(x,y)$ as coordinates on $\W$. Let $\proj_x: \C^2 \rightarrow \C$ (respectively, $\proj_y: \C^2 \rightarrow \C$) be projection of $\C^2$ onto the
$x$-coordinate (respectively, $y$-coordinate).  
\begin{figure}[ht] 
   \centering
   \includegraphics[width=6.5in]{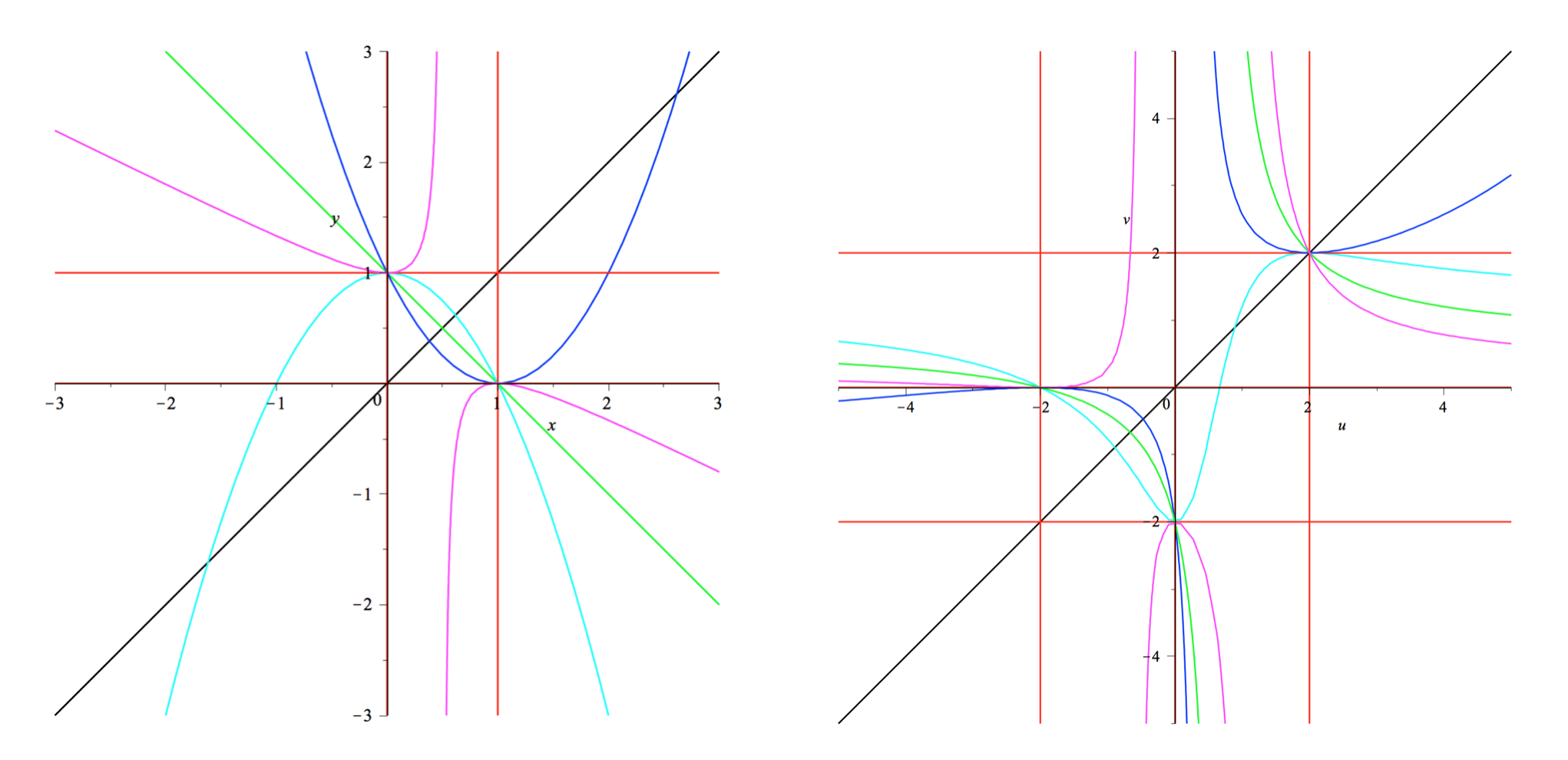} 
   \caption{On the left is the space $\W$ drawn in $\mathbb{R}^2$ in $(x,y)$-coordinates; it is the complement of the curves in $\Delta$, which are drawn in color. The black diagonal line is $\V$, and it intersects $\Delta$ in $10$ points (two of which are complex conjugates and one which is at $(\infty,\infty)$). On the right is a picture of $\W$ near $(x,y)=(\infty,\infty)$ drawn in $(u,v)$-coordinates, where $x=\frac{2u}{u-2}$ and $y=\frac{2v}{v-2}$. }
   \label{hi}
\end{figure}

\begin{proposition}\label{maps-wrt-coords}
The maps $\mu_A$, $\mu_A^B$ and $\mu_B$ can be expressed in these coordinates as follows:
the map $\mu_A$ is given by
\begin{eqnarray*}
\mu_A: \W &\rightarrow&\M_A\\
(x,y) &\mapsto& x;
\end{eqnarray*}
the map $\mu_B$ is a degree $4$ covering map,
\begin{eqnarray*}
\mu_B: \W&\rightarrow&\M_B\\
(x,y) &\mapsto& (y,z)
\end{eqnarray*}
where
$$
z=\frac{(-x^2+y+2x-1)^2}{4x(y-1+x)(1-x)};
$$
and the map $\mu_A^B$ is given by
\begin{eqnarray*}
\mu_A^B: \M_B &\to&\M_A\\
(y,z) &\mapsto& y.
\end{eqnarray*}
Thus we have
$$
\mu_A = \proj_x|_{\W} \qquad \mu_A^B \circ \mu_B = \proj_y|_{\W}.
$$
And in these coordinates, $\bp_\V=(3/4,3/4)$. 
\end{proposition}

\begin{proof} All of these assertions follow directly from the definitions.
\end{proof}

\subsection{The space $\V$}\label{spaceV}  We establish some properties of $\V$ in our particular setting.
\begin{proposition}
In these coordinates, the space $\V\subseteq \W$ is equal to the diagonal; that is, 
\[
\V=\{(x,y)\in\W\; |\; x=y\},
\]
and $\V$ is isomorphic to $\mathrm{Per}_4(0)^*$. 
\end{proposition}
\begin{proof}
For purposes of this proof, set $D:=\{(x,y)\in\W\; |\; x=y\}$. By Line (\ref{diag}), $\V\subseteq D$. Let $\gamma:[0,1]\to  D$ be a path with the property that $\gamma(0)=\bp_\V$. Because $\omega\circ q:\T_B\to \W$ is a covering map, there is a unique lift $\widetilde \gamma:[0,1]\to\T_B$ with $\widetilde \gamma(0)=\bp$. We prove that $\widetilde\gamma(t)\in \Def_B^A(f)$
for all $t \in [0,1]$, establishing the result. Let $\phi_t:\P^1\to\P^1$ be a homeomorphism representing $\widetilde\gamma(t)$, which satisfies 
\[
\phi_t|_{\{0,1,\infty\}}=\mathrm{id}|_{\{0,1,\infty\}}.
\]
There is a homeomorphism $\psi_t:\P^1\to\P^1$ representing $\sigma_f(\widetilde\gamma(t))$, which satisfies 
\[
\psi_t|_{\{0,1,\infty\}}=\mathrm{id}|_{\{0,1,\infty\}},
\]
and a rational map  $F_t:(\P^1,\psi_t(A))\to (\P^1,\phi_t(B))$ such that the diagram 
\[
\xymatrix{
(\P^1,A) \ar[d]^f \ar[rr]^{\psi_t} & & (\P^1,\psi_t(B))\ar[d]^{F_t}\\
(\P^1,B) \ar[rr]^{\phi_t}  & &(\P^1,\phi_t(B))
}
\]
commutes. The path $\widetilde \gamma$ defines an isotopy from $\phi_t:\P^1\to\P^1$ to the identity, and the path $\sigma_f\circ \widetilde\gamma$ defines an isotopy from $\psi_t:\P^1\to\P^1$ to the identity.  The composition $\phi^{-1}_t\circ \psi_t :\P^1\to\P^1$ is isotopic to the identity relative to $A$ because $\gamma(t)\in  D$ and therefore $\psi_t|_A=\phi_t|_A$. This implies  $[\psi_t]=\sigma_A^B([\phi_t])$; that is,  $\sigma_A^B(\widetilde\gamma(t))=\sigma_f(\widetilde\gamma(t))$. 

It follows by construction that $D=\mathrm{Per}_4(0)^*$. 
\end{proof}
\begin{corollary}\label{Vinjects} The space $\V$ is connected, and 
the restrictions
\[
\mu_B|_{\V}:\V\to \M_B\quad\text{and}\quad \mu_A|_{\V} :\V\to \M_A
\]
are injective. 
\end{corollary}
\begin{proof} By the above choice of coordinates, $\V$ is isomorphic to $\P^1$
 with 10 punctures (see  Figure \ref{hi}).  It is easily verified that the restrictions $\mu_B|_{\V}$ and $\mu_A|_{\V}$ are injective.
\end{proof}
\noindent As a consequence, the group $\S_f$ is equal to the stabilizer of $\Def_B^A(f)$ for our family of examples $f:(\P^1,A)\to(\P^1,B)$ as proven in Proposition \ref{stab-prop}. 
\begin{proposition}\label{stab-prop}
The subgroup of $\mod_B$ consisting of elements which restrict to automorphisms of $\Def_B^A(f)$ is $\S_f$.
\end{proposition}
\begin{proof}
Let $\mathrm{G}_f\subseteq \mod_B$ be the subgroup of elements which restrict to automorphisms of $\Def_B^A(f)$. By Proposition~\ref{SProp}, $\S_f\subseteq \mathrm{G}_f$.
Consider the commutative diagram:
$$
\xymatrix{
\Def_B^A(f) \ar[d]\ar @{^{(}->}[r]^{\iota} &\T_B\ar[d]^q&\\
\Def_B^A(f)/\S_f\ar[d]\ar @{^{(}->}[r]^{\overline{\iota}} &\T_B/\S_f\ar[d]\ar[rd]^\omega&\\
\Def_B^A(f)/\mathrm{G}_f \ar @{^{(}->}[r] &\T_B/\mathrm{G}_f \ar[d]&\T_B/\L_f\ar[ld]^{\mu_B}\\
&\M_B&
}
$$
The map 
\[
\mu_B\circ\omega\circ\overline\iota:\Def_B^A(f)/\S_f\to\M_B
\]
is injective and hence the map
$$
\Def_B^A(f)/\S_f \rightarrow \Def_B^A(f)/\mathrm{G}_f
$$
is a homeomorphism.  Thus  $\S_f = \mathrm{G}_f$.
\end{proof}

At this point, we would like to compare $\E_f$, and $\S_f$, but the fundamental groups of the spaces involved are rather complicated. Instead, we will include the space $\W$ into a somewhat simpler space $\hW$ where the fundamental group  is easier to understand, $j:\W\hookrightarrow \hW$.  We will compare the corresponding groups $\hE:=j_*(\E_f)$ and $\hS:=j_*(\S_f)$ in $\pi_1(\hW,\bp_\V)$. 

\subsection{A simpler space}\label{simple} Define 
\[
\widehat\W:=\mathbb{C}^2-\{x=0,y=0,x=1,y=1,x+y=1\}, 
\]
the complement in $\mathbb C^2$ of the lines $C, R_1,R_2,S_1,S_2$  in Figure~\ref{bpW}, and define $\widehat\V:=\{(x,y)\in \widehat\W\;|\;y=x\}$; it is isomorphic to $\P^1 -\{0,1/2,1,\infty\}$. 
Let $j:\W\hookrightarrow\hW$ and $\widehat i : \hV \hookrightarrow \hW$ be the inclusion maps. 
By functoriality of fundamental groups, we have
$$
\hE= \widehat i_*(\pi_1(\hV,\bp_\V)).
$$
\begin{proposition}\label{j-surjective}
The map 
\[
j_*:\pi_1(\W,\bp_\V)\to\pi_1(\hW,\bp_\V)
\]
is surjective. 
\end{proposition}
\begin{proof}
Let $L\in \C^2$ be a generic line containing $\bp_\V$. By the Lefschetz Hyperplane Theorem (see Corollary~\ref{LH-thm}), $\pi_1(\W,\bp_\V)$ is generated by the image of $\pi_1(L\cap \W,\bp_\V)$ under the map induced by inclusion
$$
 L\cap\W\hookrightarrow \W,
$$ 
and $\pi_1(\hW,\bp_\V)$ is generated by the image of $\pi_1(L\cap \hW,\bp_\V)$ under the map induced by inclusion
 $$
 L\cap \hW\hookrightarrow \hW.
 $$ 
 Let $j_L : L \cap \W \rightarrow L \cap \hW$ be the restriction of $j$.
 The following diagram commutes. 
\[
\xymatrix{\pi_1(L\cap \W,\bp_\V)\ar[d]^{({j_L})_*}\ar@{->>}[r] &\pi_1(\W,\bp_\V)\ar[d]^{j_*}\\
\pi_1(L\cap \hW,\bp_\V)\ar@{->>}[r] &\pi_1(\hW,\bp_\V).}
\]
Consider the map 
\[
(j_L)_*:\pi_1(L\cap \W,\bp_\V)\to \pi_1(L\cap \hW,\bp_\V).
\]
There is a right inverse given by taking oriented loops around the points
in $L \cap \hW$ to oriented loops around the same points in $L \cap \W$, and the claim follows.
\end{proof}

\noindent
Abusing notation, denote by $\proj_x$ and $\proj_y$ the restrictions of the coordinate projections to $\hW$.

\begin{proposition}\label{surjectiveSL} We have 
\[
\hSL = \{\gamma \in \pi_1(\widehat\W, \bp_{\V}) \ | \ (\proj_x)_* (\gamma) = 
(\proj_y)_*(\gamma) \}. 
\]
\end{proposition} 
\begin{proof}
By definition, $\hSL=j_*\left(\S_f\right)$, and by Proposition~\ref{diag2}
\[
\S_f= j_*\left(\{\gamma\in \L_f\;|\; (\mu_A)_*(\gamma)=\left(\mu_A^B \circ \mu_B\right)_*(\gamma)\}\right).
\]
By Proposition~\ref{maps-wrt-coords} we have the commutative diagram
$$
\xymatrix{
\W \ar[d]^{\mu_B} \ar@{^{(}->}[r]^j&\hW\ar[d]^{\proj_y}\\
\M_B \ar[r]^{\mu_A^B} &\M_A
}
$$
which implies the equality
$$
(\proj_y \circ j)_* = (\mu_A^B \circ \mu_B)_*,
$$
and by Proposition~\ref{maps-wrt-coords}, we have
$$
(\proj_x\circ j)_* = (\mu_A)_*.
$$
The claim follows by surjectivity of $j_*$ (see Proposition~\ref{j-surjective}). 
\end{proof}

\begin{remark}\label{strategy-rem}{\em
Our strategy is to produce an element $g\in \hS$ none of whose powers
other than the identity is in $\hE$.
It then follows that 
the group generated by any nontrivial element $\gamma \in j_*^{-1}(g)\subseteq \S_f$ has infinite order, and
$$
\langle \gamma\rangle \cap \E_f = 1.
$$
This implies that the cosets
$\gamma^n\cdot \E_f \subseteq \S_f$ are distinct and hence the index of $\E_f$ in $\S_f$ is infinite.
Theorem~\ref{mainthm} then follows from Proposition~\ref{cosets-prop}. 
}
\end{remark}

\subsection{Presentations}\label{present}
In this section we present the fundamental groups of $\widehat\W$ and $\widehat \V$ using standard braid monodromy techniques
\cite{mt}, \cite{eko}.  The space $\widehat\W$ is homeomorphic to the complement of the well-studied Ceva line arrangement and other computations
of its fundamental group can be found, for example, in \cite{acl}.  We repeat the computation here in order make clear the relation
with the fundamental group of $\widehat \V$.

In the following, all loops encircling punctures are oriented by the complex structure, and
 in the figures turn in the counter-clockwise direction.

\begin{figure}[ht] 
   \centering
   \includegraphics[width=5.0in]{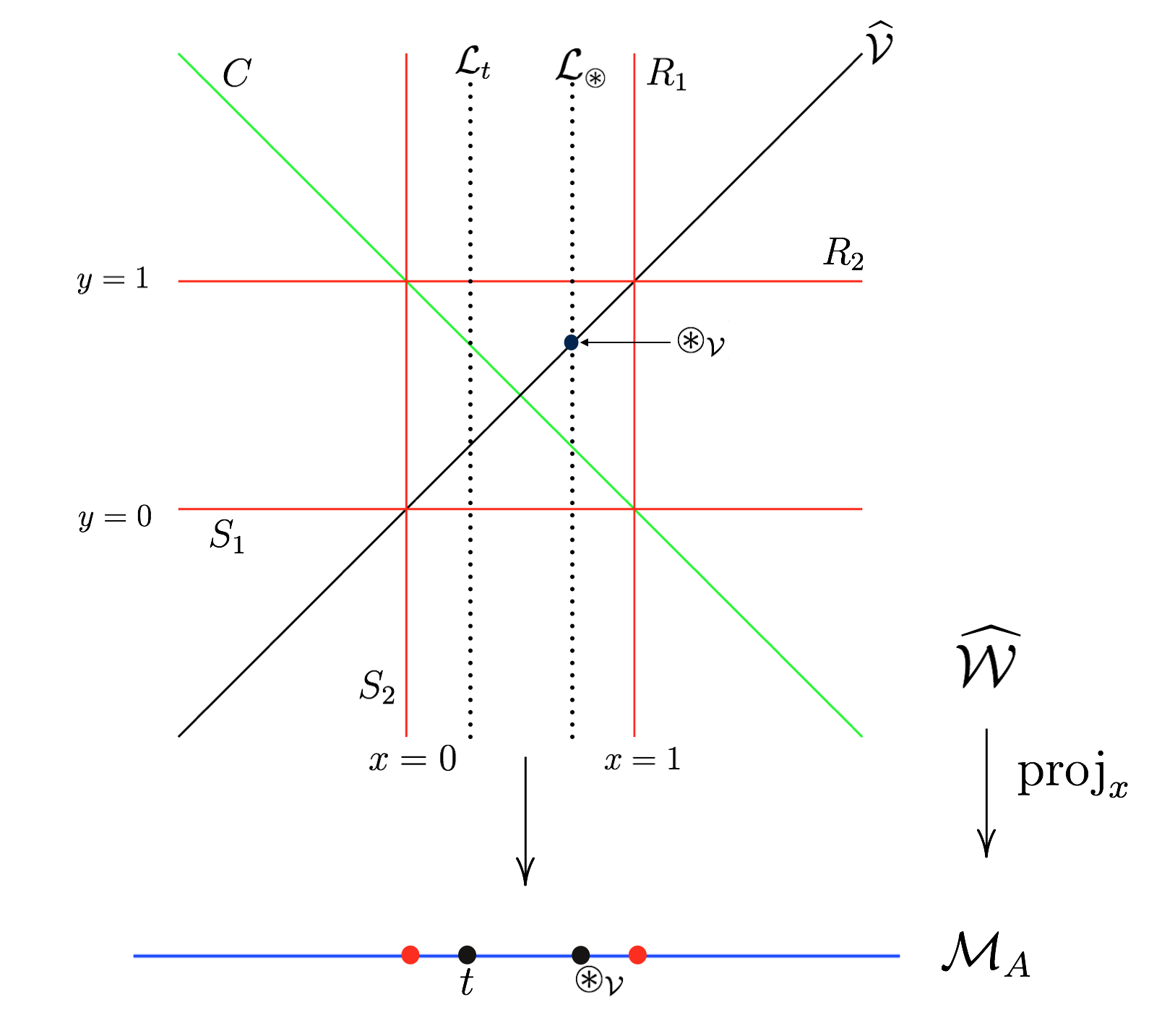} 
   \caption{The space $\widehat\W$ is a fiber bundle over $\M_A$ with each fiber equal to the complement in $\mathbb C$ of 3 points given by the intersections of the lines $S_1$, $R_2$, and $C$.}
   \label{bpW}
\end{figure}
The map $\proj_x:\hW\to \M_A$
 is a fiber bundle, over a $K(\pi,1)$. Here $\pi = \pi_1(\M_A,\bp_A)$ is the free group on
 2 generators.   Thus, for $\mathcal L_\bp := \proj_x^{-1}(\bp_A)$, we have a short 
 exact sequence of fundamental groups
\begin{eqnarray}\label{presentation-eqn}
 1 \rightarrow \pi_1(\mathcal L_\bp, \bp_\V) \rightarrow \pi_1(\widehat\W,\bp_\V) \rightarrow \pi_1(\M_A,\bp_A) \rightarrow 1.
\end{eqnarray}
We can present $\pi_1(\hW,\bp_\V)$ as a semi-direct product
$$
\pi_1(\hW,\bp_\V) = \pi_1(\mathcal L_\bp, \bp_\V) \rtimes F_2,
$$
where $F_2$ is the image of a splitting $\pi_1(\M_A,\bp_A) \hookrightarrow  \pi_1(\widehat\W,\bp_\V)$.

\begin{figure}[t] 
   \centering
   \includegraphics[width=3.5in]{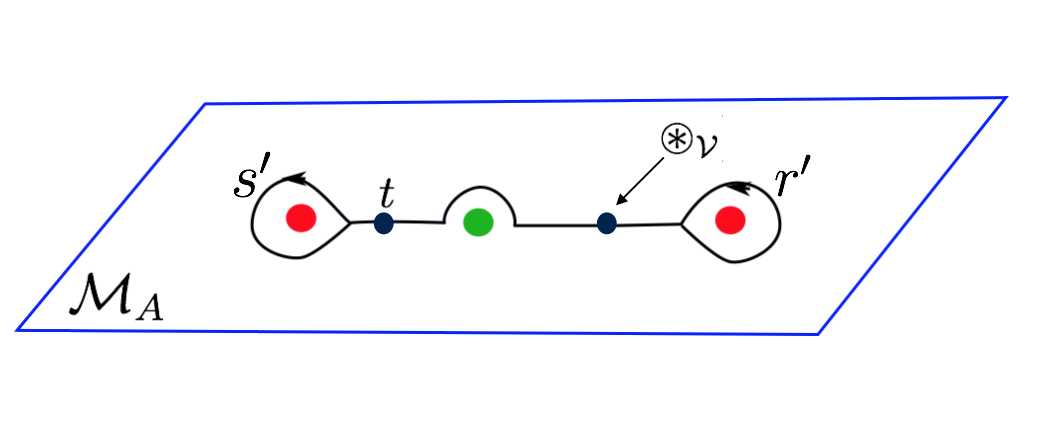} 
   \caption{Representations of the 
   generators ${r'}$ and ${s'}$ for $\pi_1(\M_A,\bp_A)$ are drawn as in black, and encircle the images of the lines $R_1$ and $S_2$.  The central dot
   indicates the image of the line $C$.  }
   \label{genbase}
\end{figure}
Present the fundamental group $\pi_1(\M_A,\bp_A)$ as the free group on the generating
loops  $r'$ and $s'$ drawn in Figure~\ref{genbase}.  The  splitting will be defined as follows. 
Note that $\proj_x |_{\hV}$ injects $\hV$ into $\M_A$. By definition, $r'$ and $s'$  lie in the
image of $\proj_x |_{\hV}$ and hence have well-defined lifts $r, s$ in $\hV$.  We take the map
that sends the generators $r'$ to $r $ and $s'$ to $s$ to be
our splitting $\pi_1(\M_A,\bp_A) \rightarrow \pi_1(\hW,\bp_\V)$.
The fiber group $\pi_1(\mathcal L_\bp,\bp_\V)$ is freely generated by $r_2,s_1,c$, the meridinal
loops on $\mathcal L_\bp$ around $R_2$, $S_1$ and $C$, respectively, as drawn in Figure~\ref{mono}.  

\begin{figure}[t] 
   \centering
   \includegraphics[width=5.5in]{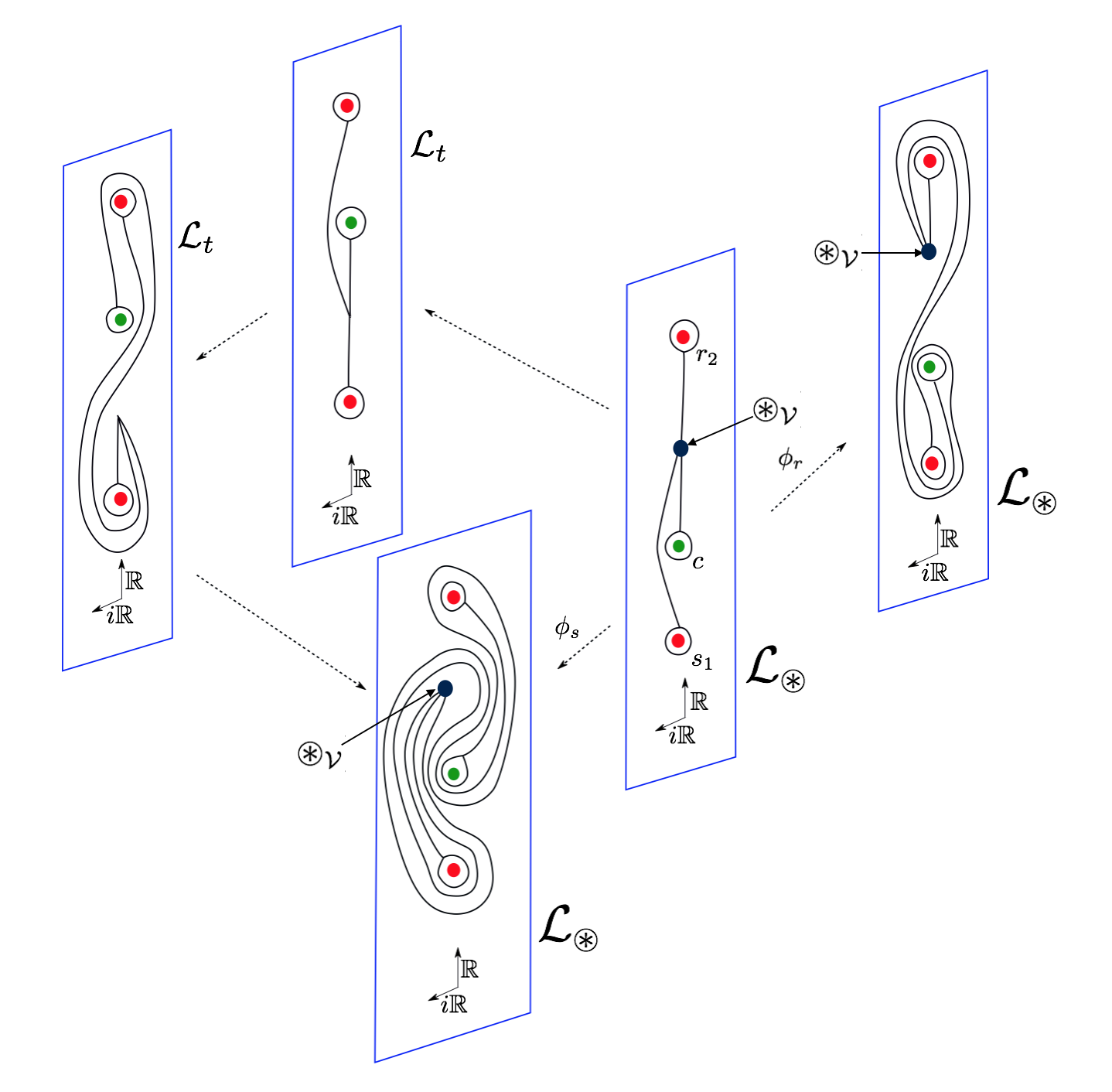} 
   \caption{The monodromy actions $\phi_r$ and $\phi_s$ on generators of $\pi_1(\mathcal L_\bp,\bp_\V)$ are drawn. The action $\phi_r$
   is split up into stages correponding to the decomposition of $r'$ into the path segments $\bp_A$ to $t$, a closed loop based at $t$
   and a return path segment $t$ to $\bp_A$ (see Figure~\ref{bpW}). }
   \label{mono}
\end{figure}

Following the Zariski-van Kampen Algorithm (see Theorem~\ref{zvk-thm}), 
we can write $\pi_1(\widehat\W, \bp_\V)$ as a semi-direct product $\langle r_2,s_1,c \rangle \rtimes \langle r, s \rangle$, 
or 
$$
\langle r_2,s_1,c,r,s : r^{-1} r_2 r  = \phi_r(r_2), r^{-1} s_1 r = \phi_r(s_1), s^{-1} r_2s = \phi_s(r_2), s^{-1}s_1s  = \phi_s(s_1)\rangle
$$
where the actions $\phi_r$ of $r$ and $\phi_s$ of  $s$ by conjugation on $\langle r_2,s_1,c \rangle$ are
defined by the monodromy of fibers over the loops $r'$ and $s'$.
To find the monodromy, we note that the local monodromy on a small loop encircling a  singular value of the projection 
in the counter-clockwise direction can be represented
by the braid that does a full counter-clockwise twist on the strands corresponding to the intersecting lines: the monodromy around a 
semi-circle in the counter-clockwise direction is a half counter-clockwise twist on the corresponding strands.
Snapshots of particular fibers over the loops $r'$ and $s'$ are illustrated in Figure~\ref{mono}.

One finds $\phi_r$ and $\phi_s$ by comparing the snapshots of $\mathcal L_\bp$ in Figure~\ref{mono}
before and after the monodromies
defined by ${r'}$ and ${s'}$.  These maps are indicated in Figure~\ref{mono} by the arrows labeled $\phi_r$ and $\phi_s$.
The maps $\phi_r$ and $\phi_s$ give the following relations for $\pi_1(\widehat\W,\bp_\V)$: 
\begin{eqnarray*}
r^{-1}r_2 r &=& r_2\\
r^{-1}cr&=& r_2^{-1}s_1cs_1^{-1}r_2\\
r^{-1}s_1r&=& r_2^{-1}s_1cs_1c^{-1}s_1^{-1}r_2\\
s^{-1}r_2s &=& s_1^{-1}cr_2c^{-1}s_1\\
s^{-1}cs &=& s_1^{-1}cr_2cr_2^{-1}c^{-1}s_1\\
s^{-1}s_1s &=& s_1.
\end{eqnarray*}

\noindent Let $r_1 = r r_2^{-1}$ and $s_2 = s s_1^{-1}$. 
With respect
to the new generators $r_1,r_2,s_1,s_2,c$, the presentation simplifies to
\begin{eqnarray}\label{presentation-eq}
&&\langle r_1,r_2,s_1,s_2,c:r_1r_2=r_2r_1, s_1s_2 = s_2s_1,
r_1s_1c = c r_1s_1 = s_1 c r_1, r_2 s_2 c = s_2 c r_2 = c r_2 s_2 \rangle.
\end{eqnarray}
The group $\pi_1(\widehat\V,\bp_\V)$ is freely generated by $r,s$ and $c$
and their images under $\widehat i_*$ are given by
\begin{eqnarray*}
r &\mapsto& r = r_1r_2\\
s &\mapsto& s  = s_1s_2\\
c &\mapsto& c.
\end{eqnarray*}

\begin{proof}[Proof of Theorem~\ref{mainthm}]
Let $g = s_2 c r_1  s_1  c r_2$.  We claim that $g \in \hS$, and no nonzero power of $g$ lies in $\hE$.
We have
$$
(\proj_x)_*(g) = s' r', \qquad (\proj_y)_*(g) = (\proj_x)_*(s_1 c r_2 s_2 c r_1) = s' r'.
$$
Consider the quotient $Q$ of $\pi_1(\hW, \bp_\V)$ given by adding the relations $a = r_1 = r_2, b = s_1 = s_2$.  Then we have
$$
Q = \langle a,b,c, d \ : d =   a b c =   b c a =  c  a  b \rangle \simeq \langle d \rangle \times \langle a, b\rangle,
$$
and quotient map $q: \pi_1(\hW, \bp_\V) \rightarrow Q$.
If we add the relation $d = 0$, then
$$
q(\hE) = \langle a^2,b^2, ab \rangle \subseteq \langle a,b \rangle,
$$
is a free subgroup, and has trivial center.

Since $q(g) = d^2$ has infinite order in  $Q$, $g$ much have infinite order in $\hS$.  On the other
hand, the nonzero powers of $q(g)$ lie in the center of $Q$, while the center of 
$\hE$ in $Q$ is trivial.  This implies
that no nonzero power of $g$ lies in $\hE$.

We have shown that $g$ satisfies the conditions in Remark~\ref{strategy-rem}
thus proving Theorem~\ref{mainthm}.
\end{proof}

\appendix

\section{Zariski-van Kampen algorithm}\label{zvk-sec}

\subsection{Fundamental groups of complements of algebraic plane curves.}
In Sections~\ref{simple} and \ref{present}, we will use a technique originally due to van Kampen and Zariski (see \cite{zvk}, \cite{chen})
for computing presentations of the fundamental group of the complement of an algebraic plane curve $\mathcal C \subseteq \C^2$.
For the reader's convenience, we give a brief outline of the technique for the special case of line arrangements defined over $\mathbb R$ 
(cf. \cite{eko}). 

Let $L = L_1 \cup \cdots \cup L_k$ and $J = J_1\cup \cdots \cup J_s \subseteq \C^2$ be unions 
of distinct lines, and let $\proj : \C^2 \rightarrow \C$ be 
a projection such that 
\begin{itemize}
\item the projection is generic with respect to $L_1,\dots,L_k$, in particular,
no component $L_i$ is equal to a  fiber of $\proj$, and 
\item each of the components $J_1,\dots,J_s$  of $J$ are fibers of $\proj$.
\end{itemize}
The Zariski-van Kampen algorithm, which we will now discribe, gives a way to compute $\pi_1(\C^2 - L\cup J,\bullet)$.

Let $U = \{u_1,\dots,u_r\} \subseteq \C$ be the images of
the intersection points of $L$, and let $V = \{v_1,\dots,v_s\} \subseteq \C$ be the images of
$J_1,\dots,J_s$ under $\mathrm{proj}$.  Let $P = U \cup V$,  $E= \C^2 - L\cup J$ and let $S = \proj^{-1}(P)$. Then $\proj$ restricts to a fiber bundle
$$
p : E - S \rightarrow \C.
$$
Let $*$ be an arbitrary point in $\C - P$, and let $F_* = p^{-1}(*)$. Then $F_*$ is isomorphic to a complex line in $\C^2$ with 
$k$ points removed.  Let $\bullet$ be a point in $F_*$.  Then, since $\C - P$ is a $K(\pi,1)$,
we have a short exact sequence 
\begin{eqnarray}\label{short}
1 \rightarrow \pi_1(F_*,\bullet) \rightarrow \pi_1(E - S,\bullet) \rightarrow \pi_1(\C-P,*) \rightarrow 1.
\end{eqnarray}
Write $\pi_1(F_*,\bullet) = \langle x_1,\dots,x_k\rangle$ and $\pi_1(\C-P,*) = \langle y'_1,\dots,y'_r, z'_1,\dots,z'_s\rangle$, 
where $y'_1,\dots,y'_r$ are meridinal loops around $u_1,\dots,u_r$ and $z'_1,\dots,z'_s$ are
meridinal loops around $v_1,\dots,v_s$.   Here a meridinal loop in $\C - P$ around a point $w \in P$ is a
loop that follows a path $\tau$ in $\C - P$ from $*$ to a point near $w$ then follows a small circle in the counter-clockwise
direction around $w$, then follows $\tau$ back to $*$.  
With this notation, $\pi_1(E - S,\bullet)$ is
a semi-direct product
$$
\pi_1(E - S,\bullet) = \pi_1(F_*,\bullet) \rtimes  \alpha_*(\pi_1(\C-P,*))
$$
where 
\begin{eqnarray*}
\alpha: \pi_1(\C-P,*) &\rightarrow& \pi_1(E - S,\bullet)\\
y'_i &\mapsto& y_i
\end{eqnarray*}
is a choice of splitting.  Thus,
$\pi_1(E - S,\bullet)$ has a presentation with generators 
$$
x_1,\dots,x_k,\ y_1,\dots,y_r,\ z_1,\dots,z_s
$$
and relations
$$
y_i^{-1}x_j y_i = \phi_i(x_j), i = 1,\dots,r,j = 1,\dots,k
$$
and
$$
z_i^{-1}x_jz_i = \psi_i(x_j),  i = 1,\dots,r,j = 1,\dots,k.
$$
where 
$$
\phi_1,\dots,\phi_r , \psi_1,\dots,\psi_s \in \text{Aut}(\pi_1(F_*,\bullet))
$$
are determined by the choice of splitting $\alpha$.  An iteration of applications of the van Kampen theorem for fundamental groups of
unions can be used to prove the following \cite{zvk} \cite{chen}.

\begin{theorem}[Zariski-van Kampen Theorem for complements of planar line arrangements] \label{zvk-thm}
If $J$ is a finite union  of fibers 
of a generic projection of $\C^2 - L$ to $\C$, then 
the presentation of $\pi_1(\C^2 - (L \cup J),\bullet)$ is obtained from that of $\pi_1(E-S,\bullet)$ by
adding the relations 
$$
y_1=1,\dots,y_k = 1.
$$
\end{theorem}

\noindent In the case when $J$ is empty, we have the following consequence.

\begin{corollary}[Lefschetz Hyperplane Theorem for complements of planar line arrangements]\label{LH-thm}  Given an arrangement of lines 
$L \subseteq \C^2$, and inclusion $F_* \hookrightarrow \C^2 - L $, where  $F_*$ is any line
in general position with respect to $L$ there is a surjective map on fundamental groups
$$
\pi_1(F_*,\bullet) \rightarrow \pi_1(\C^2-L,\bullet).
$$
where $\bullet \in F_*  - F_* \cap L$ is any arbitrary element. 
\end{corollary}

In practice, it is not difficult to find $\phi_i$ and $\psi_j$ using local monodromies, particularly in the case when the lines are defined by
real equations.  This is because we can choose generators for $\pi_1(\C - P,\ast)$ that are concatenations of real segments, and small semi-circles
or circles  centered at the points in $P$.  This allows one to decompose the monodromy into simpler pieces as is done in Section~\ref{present}.

\bigskip

\end{document}